\documentclass[10 pt, conference]{ieeeconf}  
\IEEEoverridecommandlockouts                              
\usepackage{mathtools}
\usepackage{bm}
\usepackage{color}
\usepackage{amsmath,amssymb,latexsym}
\newtheorem{theorem}{Theorem}
\newtheorem{lemma}[theorem]{Lemma}

\newtheorem{remark}[theorem]{Remark}

\usepackage{pgfplots}
\pgfplotsset{compat=newest}
\usetikzlibrary{plotmarks}
\usepgfplotslibrary{patchplots}
\usepackage{grffile}
\usepackage{amsmath}
\usepackage{algorithm}
\usepackage{algorithmic}
\newcommand{\E}[1]{\mathbb{E}\left\{ #1 \right\} }

\allowdisplaybreaks 
\usepackage{tikz}
\usetikzlibrary{graphs}
\ifnum\pdfshellescape=1
\fi
\usetikzlibrary{positioning}
\usetikzlibrary{external,arrows}

\begin{document}

\title{Remote State Estimation with Stochastic Event-triggered Sensor Schedule in the Presence of Packet Drops}

\author{ Liang Xu, Yilin Mo and Lihua Xie \thanks{ Liang Xu and Lihua Xie are with the School of Electrical and Electronic Engineering, Nanyang Technological University, Singapore (e-mail: lxu006@e.ntu.edu.sg, elhxie@ntu.edu.sg)}
\thanks{ Yilin Mo is with the Department of Automation, Tsinghua University, Beijing, China (e-mail: ylmo@tsinghua.edu.cn)}}

\maketitle

\begin{abstract}
    This paper studies the remote state estimation problem of linear time-invariant systems with stochastic event-triggered sensor schedules in the presence of packet drops between the sensor and the estimator. It is shown that the system state conditioned on the available information at the estimator side is Gaussian mixture distributed. Minimum mean square error (MMSE) estimators are subsequently derived for both open-loop and closed-loop schedules. Since the optimal estimators require exponentially increasing computation and memory, sub-optimal estimators to reduce the computational complexities are further provided. In the end, simulations are conducted to illustrate the performance of the optimal and sub-optimal estimators.
\end{abstract}

\section{Introduction}

Sensor networks have wide applications in environment and habitat monitoring, industrial automation, smart buildings, etc. In many applications, sensors are battery powered and are required to reduce the energy consumption to prolong their service life. Sensor scheduling algorithms are therefore proposed as an efficient method by scheduled transmissions to reduce the communication frequency to prolong the service time of sensor devices. Sensor scheduling algorithms can be roughly categorized as off-line schedules and event-triggered schedules. The off-line schedules are designed based on the communication frequency requirement and the statistics of the systems~\cite{YangChao2011TSP,ShiLing2011Auto,MoYilin2014SCL}. Compared with off-line schedules, event-triggered schedules depend on both the statistics and the realization of the system, which are expected to achieve better performance than off-line ones. Many triggering rules have been proposed in the literature based on the condition that, the estimation error~\cite{XiaMeng2017TAC}, error in predicated output~\cite{TrimpeSebastian2014CDC}, functions of the estimation error~\cite{WuJunfeng2013TAC,HanDuo2015TAC}, or the error covariance~\cite{TrimpeSebastian2014TAC}, exceeds a given threshold.

Wireless communications are mostly utilized in sensor networks, and packet drops are inevitable in wireless communications. Therefore, it is necessary to study how packet drops affect sensor scheduling algorithms~\cite{LeongAlexS2017TAC, MoYilin2014SCL}. It should be noted that, for off-line schedulers and estimation error covariance based event-triggered schedulers, there is no need to distinguish between the channel loss event and the hold of transmission event when designing estimators. As long as the estimator receives the packet, it can conduct the measurement update to improve the estimate and vice versa. However, the case is different for the event-triggered sensor scheduling algorithms in~\cite{WuJunfeng2013TAC, HanDuo2015TAC} where the sensor measurement is used as the trigger criterion and the hold of transmission event contains information about the sensor measurement. In the presence of possible channel losses, the estimator cannot decide whether the non-reception of the packet can be attributed to the sensor measurement or the channel loss. If it is due to that the sensor measurement lies below the given threshold, then this information can be leveraged to improve the estimate. However, if it is caused by the channel loss, the estimator will have no information about the sensor measurement and no update will be carried out. This fact complicates the optimal estimator design. Furthermore, it is proved that, in the presence of channel losses, the Gaussian properties with the stochastic event-triggered sensor scheduling algorithms in~\cite{HanDuo2015TAC} no longer hold~\cite{KungEnoch2017ACC}. 

This paper considers the same problem setting as in~\cite{HanDuo2015TAC} with the additional consideration of the presence of packet drops between the sensor and the estimator. We try to derive the MMSE estimator in the case that the estimator has no knowledge about the channel loss events and only knows the channel loss rate. We show that the conditional distributions of the system state at the estimator side are mixture Gaussian, based on which MMSE estimators are derived. Moreover, sub-optimal estimation algorithms to reduce computational complexities are provided. This paper is organized as follows. The problem formulation is given in Section~\ref{sec:ProblemFormulation}. The optimal estimators for the open loop scheduler case and the closed-loop scheduler case are studied in Section~\ref{sec:OpenLoopEstimator} and Section~\ref{sec:ClosedLoopEstimator}, respectively. Strategies to reduce the computational complexities are discussed in Section~\ref{sec:ReduceComputeComplexity}. Simulations evaluations are given in Section~\ref{sec:simulations}. This paper ends with some concluding remarks in Section~\ref{sec:Conclusion}.   

\textit{Notation}: $\mathcal{N}_x(\bar{x}, \Sigma)$ denotes the Gaussian pdf of the random variable $x$ with the mean $\bar{x}$ and the covariance matrix $\Sigma$. $f(x)$ ($\mathrm{Pr}(x)$) denotes the probability density function (probability) of the random variable $X$.
$f(x|y)$ ($\mathrm{Pr}(x|y)$) denotes the probability density function (probability) of the random variable $X$ conditioned on the event that $Y=y$.    $\mathbb{E}\{\cdot\}$ denotes the expectation operator.   $A'$, $A^{-1}$ and $|A|$ are the transpose, the inverse and the determinant of matrix $A$, respectively. The term $x'Ax$ for the symmetric matrix $A$ and vector $x$ is abbreviated as $x'A(*)$.


\section{Problem Formulation\label{sec:ProblemFormulation}}

In this paper, we are interested in the following linear dynamic system 
\begin{align*}
  x_{k+1}&=Ax_k+w_k,\\
  y_k&=Cx_k+v_k,
\end{align*}
where $x_k\in \mathbb{R}^n$, $y_k \in \mathbb{R}^m$ are the state and output; $w_k$ and $v_k$ are the process and measurement noises. We assume that $\{w_k\}_{k\ge 0}$ and $\{v_k\}_{k\ge 0}$ are white Gaussian processes with zero mean and covariance matrices $Q$ and $R$, respectively. Moreover, the initial system state satisfies $x_0\sim \mathcal{N}_{x_0}(0, \Sigma_0)$ and is independent with $w_k$ and $v_k$.

We consider the remote estimation problem where the sensor output is transmitted to the estimator through a wireless network. To reduce the communication frequency, after measuring $y_k$, the sensor follows the stochastic sensor scheduling algorithm~\cite{HanDuo2015TAC} to decide whether to transmit $y_k$ to the estimator or not. Let $s_k$ denote the decision variable by the sensor.  When $s_k=1$, the sensor transmits $y_k$ to the estimator and $s_k=0$, otherwise. We assume that the communication channel between the sensor and the estimator is a memoryless erasure channel. Let us define independent Bernoulli random variables $\gamma_k$s such that $\gamma_k=1$ if the channel is in the good state at time $k$ and $\gamma_k=0$ if otherwise. Hence, the estimator can only receive $y_k$ when both $s_k=1$ and $\gamma_k=1$. We further assume that $\Pr(\gamma_k=0)=p$ and the estimator can distinguish whether a packet arrives or not. However, in case of no packet arrival, the estimator cannot decides whether it is due to channel loss or the inactivity of the event-trigger. The following information is available to the estimator at time $k$
\begin{align}
 \mathcal{I}_k=\{s_0\gamma_0, \ldots, s_k\gamma_k, s_0\gamma_0y_0, \ldots, s_k\gamma_k y_k \} 
\end{align}
with $\mathcal{I}_{-1}=\emptyset$. The following notions are defined first and will be used in subsequent analysis. 
\begin{align*}
  \hat{x}_{k|k}&=\E{x_k|\mathcal{I}_k},& \hat{x}_{k|k-1}&=\E{x_k|\mathcal{I}_{k-1}},\\
  \hat{y}_{k|k}&=\E{y_k|\mathcal{I}_k},& \hat{y}_{k|k-1}&=\E{y_k|\mathcal{I}_{k-1}}, \\
  e_{k|k}&=x_k-\hat{x}_{k|k}, & e_{k|k-1}&= x_k-\hat{x}_{k|k-1},\\
  P_{k|k}&=\mathbb{E}\{e_{k|k}e_{k|k}'\}, & P_{k|k-1}&=\mathbb{E}\{e_{k|k-1}e_{k|k-1}'\}.
\end{align*}

The stochastic sensor scheduling algorithm~\cite{HanDuo2015TAC} operates as below. At the time $k$, the sensor randomly generates a variable $\zeta_k$ from a uniform distribution on $[0, 1]$. Then  $\zeta_k$ is compared with a function $\phi(y_k, \hat{y}_{k|k-1})$, where $\phi(y_k, \hat{y}_{k|k-1}):\mathbb{R}^m\times \mathbb{R}^m\rightarrow [0,1]$. The sensor schedules transmissions based on the following rule
\begin{align}\label{eq.StochasticSensorScheduling}
 s_k= \begin{cases}
      0, \quad \textrm{if}\; \zeta_k \le \phi(y_k, \hat{y}_{k|k-1}),  \\
      1, \quad \textrm{if}\; \zeta_k > \phi(y_k, \hat{y}_{k|k-1}).
  \end{cases}
\end{align}

Two stochastic sensor scheduling algorithms are proposed in~\cite{HanDuo2015TAC}. The first one is the open-loop scheduler, in which $\phi(y_k, \hat{y}_{k-1})=e^{-y_k'Yy_k}$ with $Y>0$. The open-loop scheduler uses the current measurement $y_k$ only to schedule the transmission. Another scheduler is the closed-loop scheduler, in which $\phi(y_k, \hat{y}_{k-1})=e^{-z_k'Zz_k}$ with $Z>0$ and $z_k=y_k-\hat{y}_{k|k-1}$. The closed-loop scheduler relies on the feedback information $\hat{y}_{k|k-1}$ to schedule the transmission. The diagram of the system is shown in Fig~\ref{fig:SystemDiagram}.
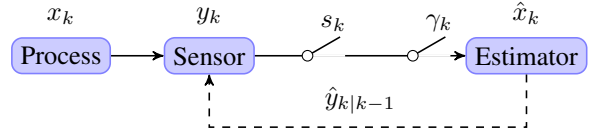
\begin{figure}[ht]
  \centering
  \begin{tikzpicture}[>=stealth', box/.style={rectangle, draw=blue!50,fill=blue!20,rounded corners, semithick}, mycircle/.style={circle, draw,inner sep=1pt}, point/.style={coordinate}] \matrix[row sep = 7mm, column sep = 7mm]{
      \node (process) [box] {Process}; &
      \node (sensor) [box] {Sensor}; &
      \node (p13) [point] {}; &
      &
      \node (p15) [point] {}; & 
      \node (estimator) [box] {Estimator};\\
      &
      \node (p22) [point] {}; &
      &
      \node (p24) [point] {};&
      &
      \node (p26) [point] {};\\
    };
    \draw[semithick,->] (process)--(sensor);
    \draw[semithick,->] (sensor)--(estimator);
    \draw[semithick,->, dashed] (estimator)--(p26)--(p22)--(sensor);

    \node (state) [above=0.05cm of process] {$x_k$};
    \node (measurement) [above=0.05cm of sensor] {$y_k$};
    \node (estimate) [above=0.05cm of estimator] {$\hat{x}_{k}$};
    \node (feedback) [above=0.05cm of p24] {$\hat{y}_{k|k-1}$};
    
    \draw[thick,white](p13)--+(00:0.5cm);
    \draw[semithick](p13)--+(30:0.5cm);
    \node  (schedulerIndicator) [above right=0.2cm and 0.05cm of p13] {$s_k$};
    \draw[fill=white](p13) circle (2pt);

    \draw[thick,white](p15)--+(00:0.5cm);
    \node  (schedulerIndicator) [above right=0.2cm and 0.05cm of p15] {$\gamma_k$};
    \draw[semithick](p15)--+(30:0.5cm);
    \draw[fill=white](p15) circle (2pt);

  \end{tikzpicture}
  \caption{Remote state estimation with stochastic sensor scheduler and packet drops
    \label{fig:SystemDiagram} }
\end{figure}

In subsequent sections, we will show that in the presence of channel losses, the distribution of $x_k$ conditioned on $\mathcal{I}_k$ is mixture Gaussian with an exponentially increasing number of components. Moreover, MMSE estimators are derived from the expectation of the mixture Gaussian distribution.

\begin{remark}  There are several reasons for not
implementing the Kalman filter at the sensor side. The first reason is that the sensor
might be primitive~\cite{HanDuo2015TAC}, so it does not have a sufficient computation capability to
run a local Kalman filter. Secondly, the system parameters might not
be available to the sensor. Thirdly, in decentralized settings
where there are multiple sensors measuring the same process, only the
fusion center which has access to all the sensor measurements can run
the Kalman filter. In the end, the state dimension might be larger
than the output dimension. Therefore, it reduces the communication
cost to transmit the sensor output and perform the Kalman filter at
the estimator side.
\end{remark}

\begin{remark}
    The closed-loop scheduler assumes perfect feedback channels and
    this is possible if the estimator has a much larger transmitting
    power than the sensor. A good example is the communication with a
    satellite. The power in the ground-to-satellite direction can be
    much larger than that in the reverse direction, so the first link can
    be considered as a (essentially) noiseless link.
\end{remark}

\section{Optimal Open-Loop Estimator\label{sec:OpenLoopEstimator}}

In this section, we assume that the open-loop stochastic sensor scheduler is applied and try to derive the MMSE estimator. First of all, the following notions are defined. For any given $i\in \mathbb{N}$ and $k\in \mathbb{N}$, let
\begin{align*}
  i_k^-=
  \begin{cases}
    i, & \textrm{if}\; i<2^k,\\
    i-2^k,& \textrm{if}\; i\ge 2^k.
  \end{cases}
\end{align*}
For any given $k\in \mathbb{N}$ and $i\in \mathbb{N}$ with  $0\le i\le 2^{k+1}-1$, define the event 
\begin{align*}
  \Gamma_{k}^i=\{(\gamma_0, \ldots, \gamma_k)=(b_0, \ldots , b_k)\}, 
\end{align*}
where $b_k$ is the $(k+1)$-th element of the binary expansion of $i$, i.e., $i=b_k2^{k}+b_{k-1}2^{k-1}+\ldots +b_0 2^0$. {
Therefore, $\Gamma_k^i$ denotes a packet drop sequence $\{\gamma_0, \ldots, \gamma_k\}$ specified by the index $i$. $i_k^{-}$ is the index of the sub-sequence $\{\gamma_0, \ldots, \gamma_{k-1}\}$ extracted from the sequence $\{\gamma_0, \ldots, \gamma_k\}$ specified by the index $i$.}

We shall prove that the posterior distribution of $x_k$ can be written as follows:
\begin{align}
  f(x_k|\mathcal{I}_k)&=\sum_{i=0}^{2^{k+1}-1}f(x_k|\Gamma_k^i, \mathcal{I}_k)\Pr(\Gamma_k^i|\mathcal{I}_k), \label{eq.estimatePDF}\\
  f(x_k|\mathcal{I}_{k-1})&=\sum_{i=0}^{2^{k}-1}f(x_k|\Gamma_{k-1}^i, \mathcal{I}_{k-1})\Pr(\Gamma_{k-1}^i|\mathcal{I}_{k-1}) .\label{eq.predictPDF}
\end{align}
One can interpret $\Gamma_k^i$ as a possible realization of the channel loss sequence up to time $k$, and $f(x_k|\Gamma_k^i, \mathcal{I}_k)$ is the pdf of $x_k$ assuming that we know the channel loss sequence is indeed $\Gamma_k^i$, which we shall prove later is indeed Gaussian. Moreover, $\Pr(\Gamma_k^i|\mathcal{I}_k)$ represents the estimator's estimate of the likelihood of channel loss sequence $\Gamma_k^i$ with the available information $\mathcal{I}_k$. Therefore, $f(x_k|\mathcal{I}_k)$ is the pdf of a Gaussian mixture. In the sequel, we will derive the expressions for  $f(x_k|\Gamma_k^i, \mathcal{I}_k)$ and $\Pr(\Gamma_k^i|\mathcal{I}_k)$. Then in view of~\eqref{eq.estimatePDF} and~\eqref{eq.predictPDF}, the optimal estimator can be obtained. 

The following results are required and are presented first. 
\begin{align}
  \label{eq:14}
  f(x_k|\Gamma_k^i, \mathcal{I}_{k-1})=      f(x_k|\Gamma_{k-1}^{i_k^-}, \mathcal{I}_{k-1}),
\end{align}
since only knowing $\gamma_k$ without knowing $y_k$ cannot help to improve the knowledge about $x_k$. 
\begin{lemma} \label{lem:OracleStatePDF}
  \begin{align*}
    f(x_k|\Gamma_k^i, \mathcal{I}_k)&=\mathcal{N}_{x_k} (m_{k|k}^i, P_{k|k}^i), \quad 0\le i \le 2^{k+1}-1\\
    f(x_k|\Gamma_{k-1}^i, \mathcal{I}_{k-1})&= \mathcal{N}_{x_k} (m_{k|k-1}^i, P_{k|k-1}^i), \quad 0\le i \le 2^{k}-1
  \end{align*}
  where $m_{k|k}^i, P_{k|k}^i, m_{k|k-1}^i, P_{k|k-1}^i$ satisfy the following recursive equations. 

  \textit{Time Update}:
  \begin{gather*}
    m_{k+1|k}^i=  Am_{k|k}^i, P_{k+1|k}^i= AP_{k|k}^iA'+Q
  \end{gather*}

  \textit{Measurement Update}:
  \begin{itemize}
  \item  For $i<2^k$,
    \begin{gather}
      m_{k|k}^i=m_{k|k-1}^i, P_{k|k}^i=P_{k|k-1}^i \label{eq:CondStatMeaUpPacketLoss2}
    \end{gather}

  \item For $i\ge 2^k$, 
  \end{itemize}
  \begin{gather}
    m_{k|k}^i= (I-K_k^{i_k^-}C)m_{k|k-1}^{i_k^-}+s_k\gamma_kK_k^{i_k^-}y_k\label{eq:CondStatMeaUpPacketReceive1}\\
    P_{k|k}^i=P_{k|k-1}^{i_k^-}-K_k^{i_k^-}CP_{k|k-1}^{i_k^-}\label{eq:CondStatMeaUpPacketReceive2}\\
    K_k^{i_k^-}=P_{k|k-1}^{i_k^-}C'[CP_{k|k-1}^{i_k^-}C'+R+(1-s_k\gamma_k)Y^{-1}]^{-1}\label{eq:CondStatMeaUpPacketReceive3}
  \end{gather}
  with  initial conditions $m_{0|-1}^0= 0, P_{0|-1}^0= \Sigma_0$.

\end{lemma}
\begin{proof} The proof of the initialization and the time-update is straightforward. The measurement update \eqref{eq:CondStatMeaUpPacketLoss2} follows from the  fact that for $i<2^k$, we have $\gamma_k=0$. Therefore, no new information is available and the measurement update is not needed. The measurement updates \eqref{eq:CondStatMeaUpPacketReceive1}, \eqref{eq:CondStatMeaUpPacketReceive2}, \eqref{eq:CondStatMeaUpPacketReceive3} follow from the fact that for $i\ge 2^k$, we have $\gamma_k=1$. Therefore, the measurement update is the same as~\cite{HanDuo2015TAC}. It should be noted that in the case $i\ge 2^k$, $s_k=\gamma_ks_k$.
\end{proof}

Next we calculate the probabilities of 
\begin{align*}
  \alpha_{k|k-1}^i=\Pr(\Gamma_k^i|\mathcal{I}_{k-1}), \alpha_{k|k}^i=\Pr(\Gamma_k^i|\mathcal{I}_{k}).
\end{align*}

\begin{lemma}\label{lem:LossProcessEstimate} $\alpha_{k|k-1}^i$ and $\alpha_{k|k}^i$ with $0\le i\le 2^{k+1}-1$ satisfy the following recursive equation.

  \textit{Time Update}:
  \begin{itemize}
  \item For $i< 2^k$,
    \begin{align}
      \label{eq:LossProcessTimeUpdate2}
      \alpha_{k|k-1}^i= p\alpha_{k-1|k-1}^i.
    \end{align}

  \item For $i\ge 2^k$, 
    \begin{align}
      \label{eq:LossProcessTimeUpdate1}
      \alpha_{k|k-1}^i =(1-p)\alpha_{k-1|k-1}^{i_k^-}.
    \end{align}
  \end{itemize}

  \textit{Measurement Update}:

  \begin{align*}
    \alpha_{k|k}^i= \frac{\Pr(s_k\gamma_k|\Gamma_k^i, \mathcal{I}_{k-1})\alpha_{k|k-1}^i}{\sum_{j=0}^{2^{k+1}-1} \Pr(s_k\gamma_k|\Gamma_{k}^j, \mathcal{I}_{k-1}) \alpha_{k|k-1}^j},
  \end{align*}
  where
  \begin{itemize}
  \item For $j< 2^k$,
    \begin{align*}
      \Pr(s_k\gamma_k|\Gamma_k^j, \mathcal{I}_{k-1})= 1-s_k\gamma_k
    \end{align*}
  \item For $j\ge 2^k$, 
    \begin{align*}
      & \Pr(s_k\gamma_k|\Gamma_k^j, \mathcal{I}_{k-1})\\
      & =s_k\gamma_k + \frac{1-2s_k\gamma_k}{\sqrt{|(CP_{k|k-1}^{j_k^-}C'+R)Y+I|}}   \\
      &\times e^{-\frac{1}{2}(Cm_{k|k-1}^{j_k^-})'[Y^{-1}+(CP_{k|k-1}^{j_k^-}C'+R)]^{-1} Cm_{k|k-1}^{j_k^-}}
    \end{align*}
  \end{itemize}
  with the initial condition $\alpha_{0|-1}^0=p, \alpha_{0|-1}^1=1-p $.
\end{lemma}
\begin{proof} \textit{Time Update}: \eqref{eq:LossProcessTimeUpdate2} follows from the fact that for $i< 2^k$, we have $\gamma_k=0$. Therefore
  \begin{align*}
    \alpha_{k|k-1}^i&=\Pr(\Gamma_{k-1}^i, \gamma_k=0|\mathcal{I}_{k-1})\\
                    &=\Pr(\gamma_k=0)\Pr(\Gamma_{k-1}^i|\mathcal{I}_{k-1})= p\alpha_{k-1|k-1}^i.
  \end{align*}

  \eqref{eq:LossProcessTimeUpdate1} follows from the fact that for $i\ge 2^k$, we have $\gamma_k=1$. Therefore
  \begin{align*}
    \alpha_{k|k-1}^i &= \Pr(\Gamma_{k-1}^{i_k^-}, \gamma_k=1|\mathcal{I}_{k-1})\\
                     &=\Pr(\gamma_k=1)\Pr(\Gamma_{k-1}^{i_k^-}|\mathcal{I}_{k-1})=(1-p)\alpha_{k-1|k-1}^{i_k^-}.
  \end{align*}

  \textit{Measurement Update}: Since
  \begin{align*}
    \alpha_{k|k}^i&=\Pr(\Gamma_k^i|\mathcal{I}_k) = \Pr(\Gamma_k^i|s_k\gamma_k, s_k\gamma_ky_k, \mathcal{I}_{k-1})\\
                  & = \Pr(\Gamma_k^i|s_k\gamma_k, \mathcal{I}_{k-1}) \\
                  &= \frac{\Pr(s_k\gamma_k|\Gamma_k^i, \mathcal{I}_{k-1})\Pr(\Gamma_k^i|\mathcal{I}_{k-1})}{\Pr(s_k\gamma_k|\mathcal{I}_{k-1})}\\
                  &= \frac{\Pr(s_k\gamma_k|\Gamma_k^i, \mathcal{I}_{k-1})\alpha_{k|k-1}^i}{\sum_{j=0}^{2^{k+1}-1} \Pr(s_k\gamma_k|\Gamma_{k}^j, \mathcal{I}_{k-1}) \alpha_{k|k-1}^j},
  \end{align*}
  where the third equality follows from the fact that when $s_k\gamma_k=0$, $s_k\gamma_ky_k=0$, it is useless to know $s_k\gamma_ky_k$; when $s_k\gamma_k=1$, knowing $s_k\gamma_ky_k$ is equivalent to know $y_k$, which is also useless in inferring $\Gamma_k^i$. Next we will show how to calculate $\Pr(s_k\gamma_k|\Gamma_k^i, \mathcal{I}_{k-1})$.

  When $i<2^k$, since $\gamma_k=0$, we have that $s_k\gamma_k\equiv 0$. Therefore  
  \begin{align}
    \label{eq:21}
    \Pr(s_k\gamma_k|\Gamma_k^i, \mathcal{I}_{k-1})= 1-s_k\gamma_k.
  \end{align}

  When $i\ge 2^k$, we have $\gamma_k=1$. Let $M_k^{i_k^-}=CP_{k|k-1}^{i_k^-}C'+R$, we then have
  \begin{align}
    & \Pr(s_k\gamma_k|\Gamma_k^i, \mathcal{I}_{k-1})  = \Pr(s_k |\Gamma_k^i, \mathcal{I}_{k-1})                                                         \nonumber \\
    & = \int_{\mathbb{R}^m}   \Pr(s_k |y_k, \Gamma_k^i, \mathcal{I}_{k-1})f(y_k|  \Gamma_k^i, \mathcal{I}_{k-1}) dy_k \nonumber \\
    & =\int_{\mathbb{R}^m}  \Pr(s_k |y_k)f(Cx_k+v_k|  \Gamma_k^i, \mathcal{I}_{k-1}) dy_k                           \nonumber \\
    & \overset{(a)}{=} \int_{\mathbb{R}^m} \left( s_k(1-2e^{-\frac{1}{2}y_k'Yy_k})+e^{-\frac{1}{2}y_k'Yy_k} \right) \nonumber \\
    &\times f(Cx_k+v_k|  \Gamma_{k-1}^{i_k^-}, \mathcal{I}_{k-1}) dy_k\nonumber \\
    & = \int_{\mathbb{R}^m}  \left( s_k(1-2e^{-\frac{1}{2}y_k'Yy_k})+e^{-\frac{1}{2}y_k'Yy_k} \right)\nonumber \\
    &\times \mathcal{N}_{y_k}(Cm_{k|k-1}^{i_k^-}, M_k^{i_k^-}) dy_k\nonumber \\
    & =s_k + (1-2s_k) \int_{\mathbb{R}^m}  e^{-\frac{1}{2}y_k'Yy_k} \mathcal{N}_{y_k}(Cm_{k|k-1}^{i_k^-}, M_k^{i_k^-}) dy_k\nonumber \\
    &= s_k+ \frac{1-2s_k }{\sqrt{(2\pi)^{m}|M_k^{i_k^-}|}}  \times e^{-\frac12 (Cm_{k|k-1}^{i_k^-})' (M_k^{i_k^-})^{-1} Cm_{k|k-1}^{i_k^-}} \nonumber \\
    & \times \int_{\mathbb{R}^m}  e^{-\frac{1}{2}y_k'[Y+(M_k^{i_k^-})^{-1}]y_k  +  (Cm_{k|k-1}^{i_k^-})'(M_k^{i_k^-})^{-1}y_k } dy_k\nonumber \\
    & \overset{(b)}{=}s_k+ \frac{1-2s_k}{\sqrt{(2\pi)^{m}|M_k^{i_k^-}|}}  e^{-\frac12 (Cm_{k|k-1}^{i_k^-})'(M_k^{i_k^-})^{-1} (*)} \nonumber \\
    &\times \sqrt{\frac{(2\pi)^m}{|Y+(M_k^{i_k^-})^{-1}|}}  \nonumber \\
    &\times e^{\frac{1}{2}(Cm_{k|k-1}^{i_k^-})'(M_k^{i_k^-})^{-1}[Y+(M_k^{i_k^-})^{-1}]^{-1}(M_k^{i_k^-})^{-1} Cm_{k|k-1}^{i_k^-}}\nonumber \\
    & \overset{(c)}{=}s_k + \frac{1-2s_k}{\sqrt{|(M_k^{i_k^-})Y+I|}}  e^{-\frac{1}{2}(Cm_{k|k-1}^{i_k^-})'[Y^{-1}+(M_k^{i_k^-})]^{-1} (*)}, \label{eq:PacketReceiveEstimation}
  \end{align}
  where $(a)$ follows from \eqref{eq:14}; $(b)$ follows from the Gaussian integral and $(c)$ follows from the matrix inversion lemma. 
\end{proof}

In view of Lemma~\ref{lem:OracleStatePDF} and Lemma~\ref{lem:LossProcessEstimate}, the MMSE estimate can be calculated by the Gaussian sum filter~\cite{Anderson1979Book} and is given as follows.
\begin{theorem}
  With the open-loop scheduler and in the presence of packet drops, the MMSE estimator is given by
\begin{align*}
  & \hat{x}_{k|k} = \sum_{i=0}^{2^{k+1}-1}    \alpha_{k|k}^i m_{k|k}^i, \\
  & P_{k|k} = \sum_{i=0}^{2^{k+1}-1}  \alpha_{k|k}^i \left( P_{k|k}^i+ (m_{k|k}^i-\hat{x}_{k|k} ) (m_{k|k}^i-\hat{x}_{k|k} )'\right),\\
  &  \hat{x}_{k|k-1} = \sum_{i=0}^{2^{k}-1}\alpha_{k-1|k-1}^im_{k|k-1}^i,  \\
  &  P_{k|k-1}= \begin{multlined}[t]
    \sum_{i=0}^{2^{k}-1}  \alpha_{k-1|k-1}^i ( P_{k|k-1}^i\\
    + (m_{k|k-1}^i-\hat{x}_{k|k-1} ) (m_{k|k-1}^i-\hat{x}_{k|k-1} )').
  \end{multlined}
\end{align*}

\end{theorem}

We can verify that when there is no packet drop, the optimal estimator degenerates to the one given in~\cite{HanDuo2015TAC}.  Besides, it is straightforward from the above expressions that the time update of the MMSE estimator can be written as
\begin{align*}
  \hat{x}_{k+1|k}=A\hat{x}_{k|k}, \quad P_{k+1|k}=AP_{k|k}A'+Q. 
\end{align*}
However, there are no such simple relations for the measurement update of the optimal estimator.

\section{Optimal Closed-Loop Estimator\label{sec:ClosedLoopEstimator}}

In this section, we derive the optimal estimators when the closed-loop scheduler is applied. Similar to the open-loop scheduler case, for closed-loop schedulers, we have the following result describing the pdf of $x_k$ conditioned on $\Gamma_k^i$ and $\mathcal{I}_k$. The proof is similar to that of Lemma~\ref{lem:OracleStatePDF} and is omitted for brevity. 
\begin{lemma} \label{lem:OracleStatePDFClosedLoop}
  \begin{align*}
    f(x_k|\Gamma_k^i, \mathcal{I}_k)&=\mathcal{N}_{x_k} (m_{k|k}^i, P_{k|k}^i), \quad 0\le i \le 2^{k+1}-1\\
    f(x_k|\Gamma_{k-1}^i, \mathcal{I}_{k-1})&= \mathcal{N}_{x_k} (m_{k|k-1}^i, P_{k|k-1}^i), \quad 0\le i \le 2^{k}-1
  \end{align*}
  where $m_{k|k}^i, P_{k|k}^i, m_{k|k-1}^i, P_{k|k-1}^i$ satisfy the following recursive equations. 

  \textit{Time Update}:
  \begin{gather*}
    m_{k+1|k}^i=  Am_{k|k}^i, P_{k+1|k}^i= AP_{k|k}^iA'+Q.
  \end{gather*}

  \textit{Measurement Update}:
  \begin{itemize}
  \item  For $i<2^k$,
    \begin{gather*}
      m_{k|k}^i=m_{k|k-1}^i, P_{k|k}^i=P_{k|k-1}^i. 
    \end{gather*}

  \item For $i\ge 2^k$, 
  \end{itemize}
  \begin{gather*}
    m_{k|k}^i= m_{k|k-1}^{i_k^-}+s_k\gamma_kK_k^{i_k^-}z_k,\\
    P_{k|k}^i=P_{k|k-1}^{i_k^-}-K_k^{i_k^-}CP_{k|k-1}^{i_k^-},\\
    K_k^{i_k^-}=P_{k|k-1}^{i_k^-}C'[CP_{k|k-1}^{i_k^-}C'+R+(1-s_k\gamma_k)Z^{-1}]^{-1},
  \end{gather*}
  where the initial conditions are $m_{0|-1}^0= 0, P_{0|-1}^0= \Sigma_0$.
\end{lemma}

Next we will show how to calculate
\begin{align*}
  \alpha_{k|k-1}^i=\Pr(\Gamma_k^i|\mathcal{I}_{k-1}), \alpha_{k|k}^i=\Pr(\Gamma_k^i|\mathcal{I}_{k}),
\end{align*}
for the closed-loop scheduler case.

\begin{lemma} \label{lem:LossProcessEstimateClosedLoop}
  $\alpha_{k|k-1}^i$ and $\alpha_{k|k}^i$ with $0\le i\le 2^{k+1}-1$ can be calculated recursively as

  \textit{Time Update}:\begin{itemize}
  \item For $i< 2^k$,
    \begin{align*}
      \alpha_{k|k-1}^i= p\alpha_{k-1|k-1}^i
    \end{align*}

  \item For $i\ge 2^k$, 
    \begin{align*}
      \alpha_{k|k-1}^i =(1-p)\alpha_{k-1|k-1}^{i_k^-}
    \end{align*}
  \end{itemize}

  \textit{Measurement Update}:
  \begin{align*}
    \alpha_{k|k}^i= \frac{\Pr(s_k\gamma_k|\Gamma_k^i, \mathcal{I}_{k-1})\alpha_{k|k-1}^i}{\sum_{j=0}^{2^{k+1}-1} \Pr(s_k\gamma_k|\Gamma_{k}^j, \mathcal{I}_{k-1}) \alpha_{k|k-1}^j},
  \end{align*}
  where
  \begin{itemize}
  \item For $j< 2^k$,
    \begin{align*}
      \Pr(s_k\gamma_k|\Gamma_k^j, \mathcal{I}_{k-1})= 1-s_k\gamma_k
    \end{align*}
  \item For $j\ge 2^k$, 

  \end{itemize}
  \begin{align*}
    & \Pr(s_k\gamma_k|\Gamma_k^j, \mathcal{I}_{k-1}) =s_k\gamma_k + \frac{1-2s_k\gamma_k}{\sqrt{|(CP_{k|k-1}^{j_k^-}C'+R)Z+I|}}\\
    &\times e^{-\frac{1}{2}(C(m_{k|k-1}^{j_k^-}-\hat{x}_{k|k-1}))'[Z^{-1}+(CP_{k|k-1}^{j_k^-}C'+R)]^{-1} (*) }
  \end{align*}
  with initial conditions $\alpha_{0|-1}^0=p, \alpha_{0|-1}^1=1-p$. 
\end{lemma}

\begin{proof} The proof of the time update is the same as that of Lemma~\ref{lem:LossProcessEstimate}. The measurement update only differs in the calculation of $\Pr(s_k\gamma_k|\Gamma_k^j, \mathcal{I}_{k-1})$ for $j\ge 2^k$ which is demonstrated as follows. For $j\ge 2^k$,  Let $M_k^{j_k^-}=CP_{k|k-1}^{j_k^-}C'+R$, we have
  \begin{align*}
    & \Pr(s_k\gamma_k|\Gamma_k^j, \mathcal{I}_{k-1})  = \Pr(s_k |\Gamma_k^j, \mathcal{I}_{k-1})                                                         \\
    & = \int_{\mathbb{R}^m}   \Pr(s_k |z_k, \Gamma_k^j, \mathcal{I}_{k-1})f(z_k|  \Gamma_k^j, \mathcal{I}_{k-1}) dz_k \\
    & =\int_{\mathbb{R}^m}  \Pr(s_k |z_k)f(C(x_k-\hat{x}_{k|k-1})+v_k|  \Gamma_k^j, \mathcal{I}_{k-1}) dz_k                           \\
    & =\int_{\mathbb{R}^m} \left( s_k(1-2e^{-\frac{1}{2}z_k'Zz_k})+e^{-\frac{1}{2}z_k'Zz_k} \right) \\
    &\times f(C(x_k-\hat{x}_{k|k-1})+v_k|  \Gamma_{k-1}^{j_k^-}, \mathcal{I}_{k-1}) dz_k\\
    & = \int_{\mathbb{R}^m}  \left( s_k(1-2e^{-\frac{1}{2}z_k'Zz_k})+e^{-\frac{1}{2}z_k'Zz_k} \right) \\
    & \times \mathcal{N}_{z_k}(C(m_{k|k-1}^{j_k^-}-\hat{x}_{k|k-1}), CP_{k|k-1}^{j_k^-}C'+R) dz_k\\
    &= s_k+ \frac{1-2s_k}{\sqrt{(2\pi)^{m}|CP_{k|k-1}^{j_k^-}C'+R|}}\\
    &\times \int_{\mathbb{R}^m} e^{-\frac{1}{2}z_k'Zz_k-\frac12 (z_k-C(m_{k|k-1}^{j_k^-}-\hat{x}_{k|k-1}))' (M_k^{j_k^-})^{-1} (*) }dz_k\\
    & =s_k + \frac{1-2s_k}{\sqrt{|(CP_{k|k-1}^{j_k^-}C'+R)Z+I|}} \\
    &\times e^{-\frac{1}{2}(C(m_{k|k-1}^{j_k^-}-\hat{x}_{k|k-1}))'[Z^{-1}+(CP_{k|k-1}^{j_k^-}C'+R)]^{-1} (*) },
  \end{align*}
  where the last equality can be proved similarly as~\eqref{eq:PacketReceiveEstimation}. 
\end{proof}

In view of Lemma~\ref{lem:OracleStatePDFClosedLoop} and Lemma~\ref{lem:LossProcessEstimateClosedLoop}, we have the following theorem.
\begin{theorem}
  With the closed-loop scheduler and in the presence of packet drops, the MMSE estimator is given by
\begin{align*}
  &      \hat{x}_{k|k}= \sum_{i=0}^{2^{k+1}-1}    \alpha_{k|k}^i m_{k|k}^i, \\
  &      P_{k|k}= \sum_{i=0}^{2^{k+1}-1}  \alpha_{k|k}^i \left( P_{k|k}^i+ (m_{k|k}^i-\hat{x}_{k|k} ) (m_{k|k}^i-\hat{x}_{k|k} )'\right),\\
  &       \hat{x}_{k|k-1} = \sum_{i=0}^{2^{k}-1} \alpha_{k-1|k-1}^i m_{k|k-1}^i,  \\
  &       P_{k|k-1} =\begin{multlined}[t]
    \sum_{i=0}^{2^{k}-1}  \alpha_{k-1|k-1}^i ( P_{k|k-1}^i \\+ (m_{k|k-1}^i-\hat{x}_{k|k-1} ) (m_{k|k-1}^i-\hat{x}_{k|k-1} )').
  \end{multlined}
\end{align*}
\end{theorem}

\section{Reduced Computational Complexities\label{sec:ReduceComputeComplexity}}

The problem considered in this paper is similar to the state estimation problem of Markov jump systems with unknown jump modes, where it is shown that the optimal nonlinear filter is obtained from a bank of Kalman filters, which requires exponentially increasing memory and computation with time~\cite{CostaO2005Book}. The generalized pseudo Bayes (GPB) algorithm~\cite{JafferAminG1971IS} and the interacting multiple model (IMM) algorithm~\cite{BlomHenkAP1988TAC} are two commonly used sub-optimal algorithms to overcome the computational complexities. The approximations in the GPB algorithm consist of restricting the probability density $f(x_k|\mathcal{I}_k)$ to depend on at most the last $N$ random variables $\gamma_k, \ldots, \gamma_{k-N+1}$ and approximate each hypothesis $f(x_k|\gamma_k, \ldots, \gamma_{k-N+1}, \mathcal{I}_k)$ with a Gaussian distribution. Moreover, a hypothesis merging operation is introduced at every step to prevent the increase of hypothesis numbers with time. The suboptimum procedure approaches the optimum one with increasing $N$. The number $N$ is to be chosen on the basis of the desired estimation performance subject to the constraint of the allowable storage capacity. 
The IMM estimator further exploits the timing of hypothesis merging to reduce the computational complexity. 
Owing to its excellent estimation performance and low computational cost, the IMM algorithm has been widely applied in various fields~\cite{MazorEfim1998TAC}. The principles of the GPB and the IMM algorithms can be utilized to derive sub-optimal estimators for the problem considered in this paper. The derivations are straightforward following~\cite{JafferAminG1971IS} and~\cite{SeahChzeEng2011TAES} and are omitted here. 

Another simple strategy is to directly apply the Gaussian mixture reduction algorithms~\cite{DFCrouse2011IFC} at each step of the measurement update to reduce the numbers of components in the Gaussian mixture model. This strategy has been applied to the Bayes filtering problem to reduce the computational complexities~\cite{WillsAdrianG2017ArXiv}. However, it should be noted that the performance of all the above mentioned algorithms can only be evaluated via Monte Carlo simulations and there are no systematic methods to analyze the performance.

\section{Simulations \label{sec:simulations}}

In simulations, we adopt the same system parameters as in~\cite{HanDuo2015TAC}, which are
\begin{align*}
  A=
  \begin{bmatrix}
    0.8 &\\
    & 0.95 
  \end{bmatrix}, C=[1,1], \Sigma_0= Q=
      \begin{bmatrix}
        1& \\
        & 1
      \end{bmatrix}, R=1.
\end{align*}

We only conduct simulations under the open loop scheduler setting with the optimal estimator, in conjunction with the oracle estimator, the GPB estimator and the OLSET-KF estimator in~\cite{HanDuo2015TAC}, where no packet drops are considered.  { The OLSET-KF estimator does not consider packet drops.
When the estimator fails to receive a packet, it always assumes that this is caused by the hold of transmission from the scheduler.} The oracle estimator is the optimal estimator under the assumption that the estimator knows the value of $\gamma_k$ at each step. Therefore, if the estimator fails to receive a packet, it knows whether this is caused by the channel or the scheduler. If this is caused by the channel, i.e., $\gamma_k=0$, no measurement update is conducted and vice versa. Clearly, the oracle estimator has the smallest mean square error (MSE) and can be used as a benchmark to evaluate the performance of other estimators.

 In simulations, the schedule parameter $Y$ is selected as $Y=1$ and $N=2$ is selected for the GPB estimator. We compare the performance of different estimators and we adopt Monte Carlo methods with 1000 independent experiments to evaluate the sum of MSE $\sum_{k=0}^{9}\E{\|x_k-\hat{x}_{k|k}\|^2}$ under different packet drop rates.
The simulation results are illustrated in Fig.~\ref{fig:sim1}, where the relative sum of MSE is plotted. The relative sum of MSE is defined as the sum of MSE of an estimator divided by the sum of MSE of the oracle estimator. It is clear from Fig.~\ref{fig:sim1} that the sum of MSE of the GPB estimator is close to that of the optimal estimator and is much smaller than the OLSET-KF, which shows the superior performance of the GPB estimator and also indicates the advantage of considering packet drops in the remote state estimation problem.
Moreover, in the case of $p=0$ and $p=1$, the sum of MSE of the OLSET-KF, the optimal estimator, the GPB estimator and the oracle estimator are equal. This is because when $p=0$ ($p=1$), the optimal estimator (GPB estimator) assigns zero probability to all the hypotheses with a $\gamma_k=0$ ($\gamma_k=1$). Therefore, only the hypothesis with $\gamma_k=1$ ($\gamma_k=0$) for all $k$ is preserved. As a result, the estimate of the optimal estimator (GPB estimator) is the same with the oracle estimator. Therefore, for the case that $p=0$ ($p=1$), the optimal estimator (GPB estimator) and the oracle estimator have the same sum of MSE. The recursions of the OLSET-KF and the oracle estimator are the same for the case $p=0$, where there are no packet drops. For the case that $p=1$, even though the recursions of the OLSET-KF and the oracle estimator are different, since they both start with $\hat{x}_{0|-1}=0$, their estimates would always be $\hat{x}_{k|k}=0$. Therefore, for the case that $p=0$ and $p=1$, the OLSET-KF and the oracle estimator have the same sum of MSE.

\begin{figure}
  \centering
  \includegraphics[width=0.40\textwidth]{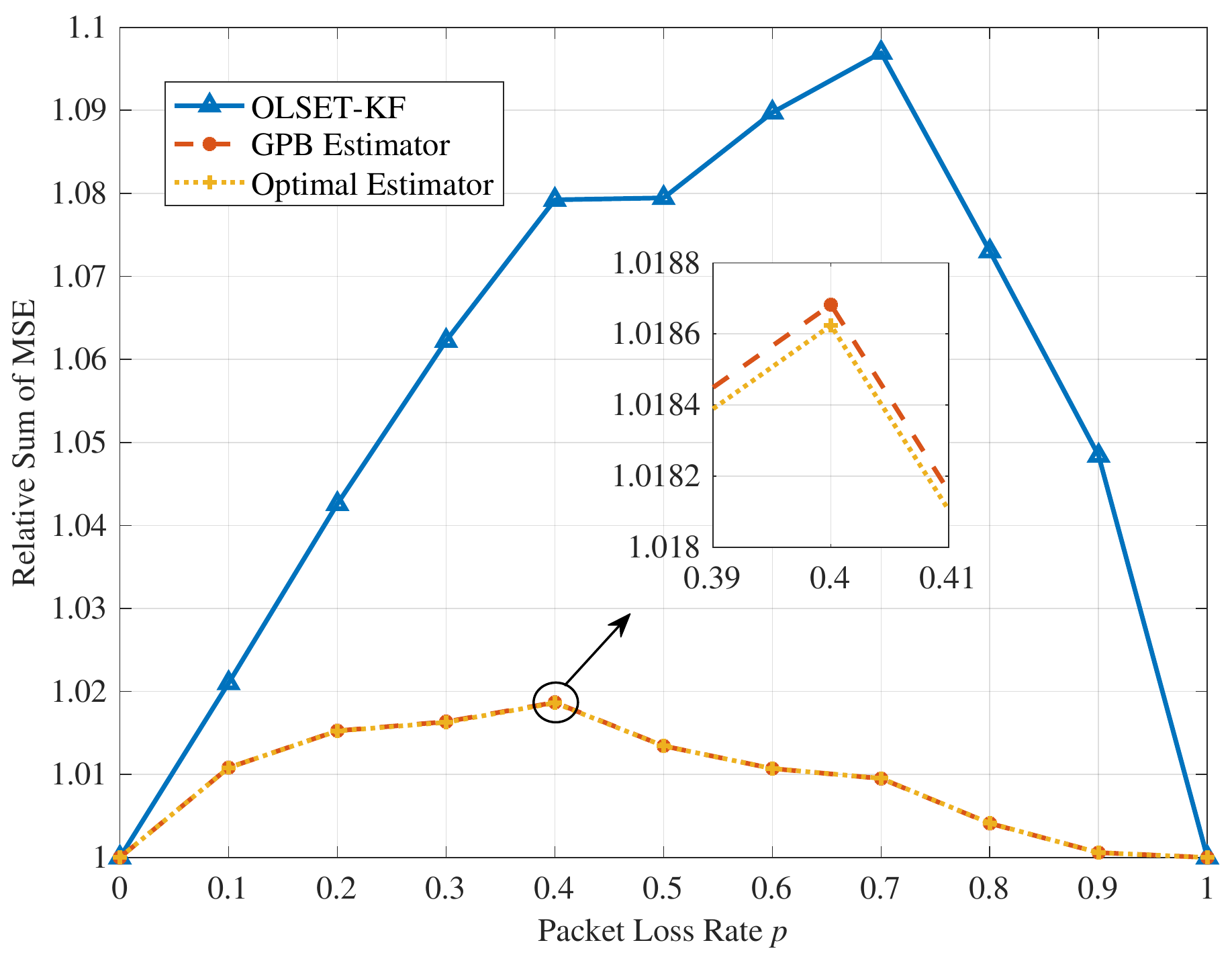}
  \caption{Relative average sum of MSE of different estimator under different packet drop rate}
  \label{fig:sim1}
\end{figure}

\section{Conclusions\label{sec:Conclusion}}

This paper studies the remote state estimation problem of linear systems with stochastic event-triggered sensor schedulers in the presence of packet drops. The conditional PDFs are computed, the optimal estimators are derived and the corresponding communication rates are analyzed. Strategies to reduce the computational complexities are discussed. However, the performance of sub-optimal estimators can only be evaluated via simulations. Sub-optimal estimators with performance guarantees are to be proposed.  

\bibliographystyle{ieeetr}

\bibliography{/Users/xuliang/Dropbox/xuliang-research/4-references/bibtex-refs_gh/references}

\end{document}